\documentclass[11pt]{article}

\usepackage{amsmath, amsfonts, amssymb}
\usepackage{mathrsfs}
\usepackage{theorem}

\usepackage{bm}
\newtheorem{thm}{Theorem}[section]
\newtheorem{prop}[thm]{Proposition}

\newtheorem{cor}[thm]{Corollary}
\newtheorem{defn}[thm]{Definition}

{\theorembodyfont{\rmfamily} \newtheorem{rem}[thm]{Remark}}
{\theorembodyfont{\rmfamily} \newtheorem{exam}{Example}[section]}

\newcommand{\ra}{\rightarrow}
\newcommand{\dis}{\displaystyle}

\def\R{\mathbb R}

\def\d{\text{\rm{d}}}
\def\E{\mathbb E}
\def\p{\mathbb P}

\def\La{\Lambda}

\def\diag{\mathrm{diag}}
\def\S{\mathcal M}
\def\rsp{\textbf{RSDP} }
\newcommand{\bxi}{\boldsymbol{\xi}}
\newcommand{\bla}{\boldsymbol{\lambda}}
\newcommand{\bbeta}{\boldsymbol{\beta}}
\newcommand{\proj}{\mathscr{I}}

\newcommand{\fin}{\hspace*{\fill}\rule{0.3em}{1ex}}
\newenvironment{proof}{{\bf \noindent Proof.}}{\fin}

\numberwithin{equation}{section}

\begin{document}

\title{Criteria for transience and recurrence of regime-switching diffusion processes\footnote{Supported in
 part by NSFC (No.11301030), NNSFC(11431014), 985-project and Beijing Higher Education Young Elite Teacher Project.}}

\author{Jinghai Shao\footnote{Email:\ shaojh@bnu.edu.cn}\\[0.6cm] {School of Mathematical Sciences, Beijing Normal University, 100875, Beijing, China}}
\maketitle
\begin{abstract}
We provide some criteria for recurrence of regime-switching diffusion processes using the theory of M-matrix and the Perron-Frobenius theorem. State-independent and state-dependent regime-switching diffusion processes in a finite space or in an infinite countable space are all studied in this work. Especially, we put forward a finite partition method to deal with switching processes in an infinite countable space.  As an application, we study the recurrence of regime-switching Ornstein-Uhlenbeck process, and provide a necessary and sufficient condition for a kind of nonlinear regime-switching diffusion processes.
\end{abstract}

\section{Introduction}
Regime-switching diffusion processes have received much attention lately, and they can provide more realistic formulation for many applications such as biology, mathematical finance, etc. See \cite{CDMR, Gh, GZ, YZ} and references therein for more details on their application.
The regime-switching diffusion process (for short, \textbf{RSDP}) studied in this work can be viewed as a number of diffusion processes modulated by a random switching device or as a diffusion process which lives in a random environment. More precisely, \rsp  is a two-component process $(X_t,\La_t)$, where $(X_t)$ describes the continuous dynamics, and $(\La_t)$ describes the random switching device. $(X_t)$ satisfies the stochastic differential equation (SDE)
\begin{equation}\label{1.1}
\d X_t=b(X_t,\La_t)\d t+\sigma(X_t,\La_t)\d B_t,\ X_0=x\in \R^d,
\end{equation}
where $(B_t)$ is a Brownian motion in $\R^d$, $d\geq 1$, $\sigma$ is $d\times d$-matrix, and $b$ is a vector in $\R^d$.
While $(\La_t)$ is a continuous time Markov chain on the state space $\S=\{1,2,\ldots,N\}$ with $N<\infty$ or $N=\infty$ satisfying
\begin{equation}\label{1.2}
\p(\La_{t+\delta}=l|\La_t=k, X_t=x)=\left\{\begin{array}{ll} q_{kl}(x)\delta+o(\delta), &\text{if}\ k\neq l,\\
                                       1+q_{kk}(x)\delta+o(\delta), & \text{if}\ k=l,
                         \end{array}\right.
\end{equation}
for $\delta>0$.
Throughout this work, the $Q$-matrix $Q_x=(q_{kl}(x))$ is assumed to be irreducible and conservative for each $x\in \R^d$, so $q_k(x)=-q_{kk}(x)=\sum_{l\neq k}q_{kl}(x)<\infty$ for every $x\in\R^d$, $k\in \S$.
If the $Q$-matrix $(q_{kl}(x))$ does not depend on $x$, then $(X_t,\La_t)$ is called a state-independent \textbf{RSDP}; otherwise, it is called a state-dependent one. When $N$ is finite, namely, $(\La_t)$ is a Markov chain on a finite state space, we call $(X_t,\La_t)$ a \rsp  in a finite state space.  When $N$ is infinite, we call $(X_t,\La_t)$ a \rsp in an infinite state space.
Next, we collect some conditions used later.

\vskip 0.3cm
\noindent (H) \  There exists a constant $\bar K>0$ such that
\begin{itemize}
\item[(i)]\ $x\mapsto q_{ij}(x)$ is a bounded continuous function for each pair of $i,j\in \S$.
\item[(ii)]\ $|b(x,i)|+\|\sigma(x,i)\| \leq \bar K(1+|x|),\quad x\in \R^d,\  i\in \S$.
\item[(iii)]\ $|b(x,i)-b(y,i)|+\|\sigma(x,i)-\sigma(y,i)\|\leq \bar K|x-y|,\quad x,\,y\in \R^d,\ i\in \S$.
\item[(iv)]\ For each $i\in \S$, $a(x,i)=\sigma(x,i)\sigma(x,i)^\ast$ is uniformly positive definite.
\end{itemize}
Here and in the sequel, $\sigma^\ast$ stands for the transpose of matrix $\sigma$, and $\|\sigma\|$ denotes the operator norm.
Hypothesis (H-i), (H-ii) and  (H-iii) guarantee the existence of a unique nonexplosive solution of (\ref{1.1}) and (\ref{1.2}) (cf. \cite[Theorem 2.1]{YZ}). Hypothesis (H-iv) is used to ensure that $(X_t,\La_t)$ possesses strong Feller property (cf. \cite{Xi08}, \cite{ZY09}), which will be used in the study of exponential ergodicity.

Corresponding to  the process $(X_t,\La_t)$, there is a family of diffusion processes defined by
\begin{equation}\label{1.3}
\d X_t^{(i)}=b(X_t^{(i)},i)\d t+\sigma(X_t^{(i)},i)\d B_t,
\end{equation}
for each $i\in\S$. These processes $(X_t^{(i)})$ $(i\in\S)$ are the diffusion processes associated with $(X_t,\La_t)$ in each fixed environment.
The recurrent behavior of $(X_t,\La_t)$ is intensively connected with its recurrent behavior in each fixed environment. But this connection is rather complicated as having been noted by \cite{PS}. In \cite{PS}, some examples in $[0,\infty)$ with reflecting boundary at $0$ and $\S=\{1,2\}$ were constructed. They showed that even when $(X_t^{(1)})$ and $(X_t^{(2)})$ are both positive recurrent (transient), $(X_t,\La_t)$ could be transient (positive recurrent, respectively) by choosing suitable transition rate $(q_{ij})$ between two states. In view of this complicatedness, it is a challenging work to determine the recurrent property of a regime-switching diffusion process. There are lots of work having been dedicated to this task. See, for instance, \cite{BGM, BLMZ, CH, PP, PS, YZ} and references therein. Besides constructing the examples we mentioned above, \cite{PS} also studied the reversible state-independent \rsp\!. In \cite{PP}, the author provided a theoretically complete characterization of recurrence and transience for a class of state-independent \rsp\!, which we will state more precisely later. In \cite{CH}, some necessary and sufficient conditions were established to justify the exponential ergodicity of state-independent and state-dependent \rsp\! in a finite state space.  The convergence in total variation norm and in Wasserstein distance were both studied in \cite{CH}. However, the cost function used in \cite{CH} to define the Wasserstein distance is bounded.  All the previously mentioned work considered only the \rsp in a finite state space. Although the general criteria by the Lyapunov functions for Markov processes still work for \rsp\!, it is well known that finding a suitable Lyapunov function is a difficult task for \rsp due to the coexistence of generators for diffusion process and jump process. So it is better to provide some easily verifiable criteria in terms of the coefficients of diffusion process $(X_t)$ and the $Q$-matrix of $(\La_t)$.
In this direction, \cite{ZY09} has provided some criteria for a class of state-dependent \rsp $(X_t,\La_t)$ in a finite state space. Precisely,   the continuous component $(X_t)$ considered in \cite{ZY09} behaves like a linear one  and $Q$-matrix $(q_{ij}(x))$ behaves like a state-independent $Q$-matrix $(\hat q_{ij})$ in a neighborhood of $\infty$.

In \cite{Sh14a}, we studied the ergodicity for \rsp in Wasserstein distance. Both state-independent and state-dependent \rsp  in finite and infinite state spaces are studied in \cite{Sh14a}.  The cost function used in \cite{Sh14a} is not necessarily bounded. We put forward some new criteria for ergodicity based on the theory of M-matrix and Perron-Frobenius theorem.  Our present work is devoted to studying the recurrent property of \rsp in the total variation norm.
Furthermore, in the present work, we also study the recurrence for \rsp in an infinite state space, which is rarely studied before. Based on the criteria given by the M-matrix theory, we put forward a finite partition method (see Theorem \ref{t-3} below).

As an application of our criteria, we develop the study in \cite{PP} and \cite{ZY09}. In \cite{PP}, the authors considered
the state-independent \rsp $(X_t,\La_t)$ in $\R^d\times \S$ with $d\geq 2$ and $\S$ a finite set. For each $i\in \S$, the associated diffusion  $(X_t^{(i)})$ has the infinitesimal generator $L^{(i)}=\frac 12\Delta+\Theta$, where
\begin{equation}\label{V-pp}
\Theta(x,i)=|x|^\delta\hat b(x/|x|,i)\cdot \nabla,\quad \delta\in [-1,1).
\end{equation}
Let $S^{d-1}$ denote the $d-1$-dimension sphere, and $\mu$ be the invariant probability measure for $(\La_t)$. 
In \cite{PP}, they studied the process under the condition  that $\hat b(\phi,i)\not\equiv 0,\ \hat b(\phi,i)\in C^1(S^{d-1})$ for each $i\in \S$, and
\begin{equation}\label{b-pp}
\sum_{i\in \S}\hat b(\phi,i)\mu_i=0\quad \text{for each $\phi\in S^{d-1}$.}
\end{equation}
Condition (\ref{b-pp}) allows them to transform the problem into studying the recurrent behavior of the generator
\[\hat L=r^\gamma\big[c_1(\phi)\frac{\partial^2}{\partial r^2}+\frac{c_2(\phi)}{r}\frac{\partial}{\partial r}+\frac 1r\frac{\partial}{\partial r}D_{S^{d-1}}+\frac 1{r^2}L_{S^{d-1}}\big],
\]
where $\gamma=0$, if $-1\leq \delta\leq 0$, and $\gamma=2\delta$, if $0<\delta<1$, $c_1(\phi)\geq 0$, $D_{S^{d-1}}$ is a first-order operator on $S^{d-1}$ and $L_{S^{d-1}}$ is a (possible degenerate) diffusion generator on $S^{d-1}$. By posing some further conditions on $c_1(\phi)$ and $c_2(\phi)$, they got a quantity $\rho$ expressed in terms of $c_1(\phi)$, $c_2(\phi)$ and the density of invariant probability measure of the process corresponding to $\hat L$. They showed that $(X_t,\La_t)$ is recurrent or transient according to whether $\rho\leq 0$ or $\rho>0$. Theoretically, this result is complete although calculating $\rho$ is a difficult task, which has been pointed out in \cite{PP}. In this work, roughly speaking, we consider the processes corresponding to $\sum_{i\in\S}  \mu_i\hat b_i(\phi,i)\neq 0$.

The usefulness and sharpness of the criteria established in this work can be seen from the following example.
Let $(\La_t)$ be a continuous time Markov chain on $\{1,2,\ldots,N\}$, $N<\infty$, equipped with an irreducible conservative $Q$-matrix $(q_{ij})$.  Let  $\mu$ be the invariant probability measure of $(\La_t)$. Let $(X_t)$ be a random diffusion on $[0,\infty)$ with reflecting boundary at $0$ satisfying
\[\d X_t=b_{\La_t}X_t^\delta\d t+\d B_t, \quad \delta\in [-1,1].\]
In the case $\delta\in [-1,1)$, if $\sum_{i=1}^N\mu_ib_i\leq 0$, then $(X_t,\La_t)$ is recurrent; if $\sum_{i=1}^N\mu_ib_i>0$, then $(X_t,\La_t)$ is transient.
In the case $\delta=1$, if $\sum_{i=1}^N\mu_ib_i<0$, then $(X_t,\La_t)$ is exponentially ergodic; if $\sum_{i=1}^N\mu_i b_i>0$, then $(X_t,\La_t)$ is transient. Note that the case $\sum_{i=1}^N \mu_i b_i=0$ has been studied by Corollary 2 of \cite{PP} and Remark 2 following it.

This work is organized as follows. In Section 2, we provide the first type of criteria for recurrence of \textbf{RSDP}, which uses a common function to measure the recurrent behaviour of \rsp in each fixed environment. Then these criteria are applied to study the recurrence of regime-switching Ornstein-Uhlenbeck processes. In Section 3, we provide the second type of criteria, which uses a couple of functions to measure the recurrent behaviour of \rsp in each fixed environment. Then we apply these criteria to study the processes considered in \cite{PP}.

\section{Criteria for recurrence and transience: I}

Let $(X_t,\La_t)$ be defined by (\ref{1.1}) and (\ref{1.2}). Its corresponding diffusion $(X_t^{(i)})$ in each fixed environment $i\in\S$
is defined by (\ref{1.3}), and the generator $L^{(i)}$  of $(X_t^{(i)})$ is given by
\[L^{(i)}=\frac 12\sum_{k,l=1}^d a^{(i)}_{kl}(x)\frac{\partial^2}{\partial x_k\partial x_l}+\sum_{k=1}^d b_k^{(i)}(x)\frac{\partial}{\partial x_k},\]
where
$a^{(i)}(x)=\sigma(x,i)\sigma(x,i)^\ast$, $b^{(i)}(x)=b(x,i)$.
For a vector $\bbeta=(\beta_1,\ldots,\beta_N)^\ast$, we use $\diag(\bbeta)=\diag(\beta_1,\ldots,\beta_N)$ to denote the diagonal matrix generated by  $\bbeta$ as usual.

Our first type of criteria for recurrence of \rsp uses a common function $V\in C^2(\R^d)$ to measure the recurrent property of the corresponding diffusion in each fixed environment. Then combine it  with the recurrent behavior of Markov chain to determine the recurrent property of \rsp $(X_t,\La_t)$. Let  $V\in C^2(\R^d)$  satisfy  the following condition:
\begin{itemize}
\item[(A1)] There exist constants $r_0>0$ and $\beta_i\in \R$, $i\in\S$ such that
\begin{equation*}\label{cond-L}
V(x)>0,\qquad L^{(i)} V(x)\leq \beta_i V(x),\quad |x|> r_0.
\end{equation*}
\end{itemize}
Here the constant $\beta_i$ could be negative or positive, which represents the recurrent behavior of $(X_t^{(i)})$ under the measurement tool $V$. Then using the Perron-Frobenius theorem, we get our first criterion for recurrence of state-independent \rsp in a finite state space.

\begin{thm}\label{t-2}
Let $(X_t,\La_t)$ be a state-independent \rsp defined by (\ref{1.1}), (\ref{1.2}) with $N<\infty$.
Assume that (H) holds and there exists a function $V\in C^2(\R^d)$ such that condition (A1) holds and
\begin{equation}\label{mu-beta}
\sum_{i\in\S}\mu_i\beta_i<0,
\end{equation} where  $\mu=(\mu_i)_{i\in \S}$ is the invariant probability measure of $(\La_t)$.
Then $(X_t,\La_t)$ is transient if\, $\lim_{|x|\ra \infty} V(x)=0$, and is exponentially ergodic if \,$\lim_{|x|\ra \infty} V(x)=\infty$.
\end{thm}

\begin{proof}
Let $Q_p=Q+p\,\diag(\bbeta)$, $p>0$, and
\[\eta_p=-\max_{\gamma\in \mathrm{spec}(Q_p)}\mathrm{Re}\,\gamma, \quad \text{where $\mathrm{spec}(Q_p)$ denotes the spectrum of $Q_p$}.\]
Let $Q_{(p,t)}=e^{tQ_p}$, then the spectral radius $\mathrm{Ria}(Q_{(p,t)})$ of $Q_{(p,t)}$ equals to $e^{-\eta_p t}$. Since all coefficients of $Q_{(p,t)}$ are positive (see the argument of \cite[Proposition 4.1]{BGM} for details), the Perron-Frobenius theorem (see \cite[Chapter 2]{BP}) yields $-\eta_p$ is a simple eigenvalue of $Q_p$. Moreover, note that the eigenvector of $Q_{(p,t)}$ corresponding to $e^{-\eta_p t}$ is also an eigenvector of $Q_p$ corresponding to $-\eta_p$. Then Perron-Frobenius theorem ensures that there exists an eigenvector $\bxi\gg 0$ of $Q_p$ associated with the eigenvalue $-\eta_p$. Now applying Proposition 4.2 of \cite{BGM} (by replacing $A_p$ there with $Q_p$ and changing the sign of $p$), if $\sum_{i=1}^N\mu_i\beta_i<0$, then there exists some $p_0>0$ such that $\eta_p>0$ for any $0<p<p_0$. Fix a $p$ with $0<p<\min\{1,p_0\}$ and an eigenvector $\bxi\gg 0$, then we obtain
\[Q_p\,\bxi=(Q+p\,\diag(\bbeta ))\bxi=-\eta_p\,\bxi\ll 0.\]
Put $f(x,i)=V(x)^p\xi_i$, $x\in \R^d$, $i\in\S$. For $|x|>r_0$, $i\in \S$, due to \cite{Sko},
\begin{equation}\label{e-1}
\begin{split}
\mathscr A f(x,i)&=Q\bxi(i) V(x)^p+\xi_iL^{(i)}V(x)^p\\
&\leq \big(Q\bxi(i)+p\beta_i\xi_i\big)V(x)^p\\ &=-\eta_p\,\xi_iV(x)^p=-\eta_p f(x,i).
\end{split}
\end{equation}
Therefore, according to the Foster-Lyapunov drift conditions (cf. \cite[Section 2, p.443]{PP} or \cite[Theorem 3.26]{YZ}),  we obtain that
$(X_t,\La_t)$ is positive recurrent if $\lim_{|x|\ra \infty} V(x)=\infty$, and $(X_t,\La_t)$ is transient if $\lim_{|x|\ra \infty}V(x)=0$. Moreover, according to \cite[Theorem 5.1]{Xi08}, inequality (\ref{e-1}) yields that $(X_t,\La_t)$ is exponentially ergodic.
\end{proof}
\begin{rem} We give a heuristic explanation of the condition (\ref{mu-beta}) in previous theorem.
As $\mu$ is the invariant probability measure of $(\La_t)$, $\mu_i$ represents in some sense the time ratio spent by $(\La_t)$ in the state $i$. $\beta_i$ represents the recurrent behavior of $(X_t^{(i)})$. Therefore, the quantity $\sum_{i\in\S}\mu_i\beta_i$ averages the recurrent behavior of $(X_t^{(i)})$ with respect to $\mu$, which determine the recurrent behavior of $(X_t,\La_t)$ according to previous theorem.
\end{rem}

Next, we shall use the theory of M-matrix to provide a criterion on recurrence of state-independent \rsp in a finite state space. This criterion can be extended to deal with state-dependent \rsp in a finite state space or state-independent \rsp in an infinite state space.
Let us introduce some notation and basic properties on M-matrix.  We refer the reader to \cite{BP} for more discussion on this topic.

Let $B$ be a matrix or vector. By $B\geq 0$ we mean that all elements of $B$ are non-negative.  By $B\gg 0$, we mean that all elements of $B$ are positive.
\begin{defn}[M-matrix] A square matrix $A=(a_{ij})_{n\times n}$ is called an M-Matrix if $A$ can be expressed in the form $A=sI-B$ with some $B\geq 0$ and $s\geq\mathrm{Ria}(B)$, where $I $ is the $n\times n$ identity matrix, and $\mathrm{Ria}(B)$ the spectral radius of $B$. When $s>\mathrm{Ria}(B)$,   $A$ is called a nonsingular M-matrix.
\end{defn}
We cite some conditions equivalent to that $A$ is a nonsingular M-matrix as follows, and refer to  \cite{BP} for more discussion on this topic.
\begin{prop}[\cite{BP}]\label{m-matrix}
The following statements are equivalent.
\begin{enumerate}
\item $A$ is a nonsingular $n\times n$ M-matrix.
\item All of the principal minors of $A$ are positive; that is,
\[\begin{vmatrix} a_{11}&\ldots&a_{1k}\\ \vdots& &\vdots\\ a_{1k}&\ldots&a_{kk}\end{vmatrix}>0 \ \  \text{for every $k=1,2,\ldots,n$}.\]
\item Every real eigenvalue of $A$ is positive.
\item $A$ is semipositive; that is, there exists $x\gg 0$ in $\R^n$ such that $Ax\gg0$.
\end{enumerate}
\end{prop}

\begin{thm}\label{t-1}
Let $(X_t,\La_t)$ be a state-independent \rsp in $\S$ with $N<\infty$.
Assume that (H) holds and there exists a function $V\in C^2(\R^d)$ such that condition (A1) is satisfied and the matrix $-\big(Q+\diag(\bbeta )\big)$ is a nonsingular M-matrix. Then $(X_t,\La_t)$ is exponentially ergodic if\, $\lim_{|x|\ra\infty} V(x)=\infty$ and is transient if \,$\lim_{|x|\ra \infty}V(x)=0$.
\end{thm}

\begin{proof}
Denote by $\mathscr A$ the generator of $(X_t,\La_t)$. Due to \cite{Sko},
\[\mathscr A f(x,i)=L^{(i)}f(\cdot,i)(x)+Qf(x,\cdot)(i),\]
where $Qg(i)=\sum_{j\neq i}q_{ij}(g_j-g_i)$ for $g\in \mathscr B(\S)$. As $-\big(Q+\diag(\bbeta)\big)$ is a nonsingular M-matrix, by Proposition \ref{m-matrix}, there exists a vector $\bxi=(\xi_1,\ldots,\xi_N)^\ast\gg 0$ such that
\[\boldsymbol{\lambda}=(\lambda_1,\ldots,\lambda_N)^\ast=-\big(Q+\diag(\bbeta)\big)\bxi\gg 0.\]
Take $f(x,i)=V(x)\xi_i,\quad x\in \R^d,\, i\in \S$, then for $|x|>r_0$, $i\in \S$,
\begin{equation}\label{ine-1}
\begin{split}
\mathscr A f(x,i)=& Q \bxi(i) V(x)+\xi_iL^{(i)} V(x)\\
      &\leq \big(Q\bxi(i)+\beta_i\xi_i\big)V(x) =-\lambda_i V(x)\\
      &=-\frac{\lambda_i}{\xi_i} f(x,i) \leq -\min_{1\leq i\leq N}\Big(\frac{\lambda_i}{\xi_i}\Big) f(x,i).
\end{split}
\end{equation}
As $N<\infty$, we have $\min_{1\leq i\leq N}(\lambda_i/\xi_i)>0$.
Then analogous to the argument of Theorem \ref{t-2}, we can conclude the proof.
\end{proof}

Now we proceed to study the recurrence of the state-dependent \rsp in a finite state space.
To this aim, we need to introduce an auxiliary Markov chain $(\tilde \La_t)$ on $\S$ with a conservative $Q$-matrix defined by:
\begin{equation}\label{tilde-q}
\tilde q_{ik} =\left\{\begin{array}{ll}\sup_{x\in\R^d}   q_{ik}(x) \ &\text{if $k<i$},\\
  \inf_{x\in \R^d} q_{ik}(x) \ &\text{if $k>i$},
 \end{array}\right.\ \ \text{and}\ \tilde q_{ii} =-\sum_{k\neq i}\tilde q_{ik}, \quad i\in \S.
\end{equation}

\begin{thm}\label{t-finite}
  Let $(X_t,\La_t)$ be a state-dependent \rsp in $\S$ with $N<\infty$. Assume that (H) holds, and there exists a function $V\in C^2(\R^d)$ such that condition (A1) is satisfied and the matrix $-\big(\tilde Q+\diag(\bbeta )\big)H_N$ is a nonsingular M-matrix, where $\tilde Q=(\tilde q_{ij})$ is defined by (\ref{tilde-q}) and
  \begin{equation}
    H_N=\begin{pmatrix}
      1&1&1&\cdots&1\\
      0&1&1&\cdots&1\\
      0&0&1&\cdots&1\\
      \vdots&\vdots&\vdots&\cdots&\vdots\\
     0&0&0&\cdots&1
      \end{pmatrix}_{N\times N}.
  \end{equation}
  Then $(X_t,\La_t)$ is transient if\, $\lim_{|x|\ra \infty} V(x)=0$ and is exponentially ergodic if\, $\lim_{|x|\ra\infty} V(x)=\infty$.
\end{thm}

\begin{proof}
Since $-(\tilde Q+\diag(\bbeta))H_N$ is a nonsingular M-matrix, by Proposition \ref{m-matrix}, there exists a vector $\eta\gg 0$ such that
\[\bm{\lambda}=-(\tilde Q+\diag(\bbeta))H_N\eta\gg 0.\]
Set $\xi=H_N\eta$, then
\[\xi_i=\eta_i+\cdots+\eta_N\quad \text{for}\ i=1,\ldots,N.\]
The strict positiveness of $\eta$ implies that $\xi_{i+1}<\xi_i$ for $i=1,\ldots,N-1$, and $\xi\gg 0$. By the definition of $(\tilde q_{ij})$, we obtain that for every $i\in \S$, $x\in\R^d$,
\begin{align*}
  Q_x\xi(i)&=\sum_{j>i}q_{ij}(x)(\xi_j-\xi_i)+\sum_{j<i}q_{ij}(x)(\xi_j-\xi_i)\\
  &\leq \sum_{j>i}\tilde q_{ij} (\xi_j-\xi_i)+\sum_{j<i}\tilde q_{ij}(\xi_j-\xi_i).
\end{align*}
Consequently, by setting $f(x,i)=V(x)\xi_i$ for $x\in\R^d$, $i\in\S$, we get
\begin{align*}
  \mathscr Af(x,i)&=Q_x\xi(i) V(x)+\xi_iL^{(i)} V(x)\\
  &\leq \big(\tilde Q \xi(i)+\beta_i\xi_i\big)V(x)=-\lambda_i V(x)\leq 0.
\end{align*}
  Then analogous to the argument of Theorem \ref{t-2}, we can conclude the proof.
\end{proof}

Now we extend Theorem \ref{t-1} to deal with state-independent \rsp in an infinite state space.
Let $V\in C^2(\R^d)$ such that (A1) holds and $\bar K=\sup_{i\in \S}\beta_i<\infty$.
As the M-matrix theory is about matrices with finite size, we shall put forward a finite partition method to transform the \rsp in an infinite state space into a new \rsp in  a finite state space. Let
\[\Gamma=\{-\infty=k_0<k_1<\ldots<k_{m-1}<k_m=\bar K\}\] be a finite partition of $(-\infty, \bar K]$. Corresponding to $\Gamma$, there exists a finite partition
$F=\{F_1,\ldots,F_m\}$ of $\S$ defined by
\[F_i=\{j\in \S;\ \beta_j\in (k_{i-1},k_i]\},\quad i=1,2,\ldots,m.\]
We assume each $F_i$ is nonempty, otherwise, we can delete some points in the partition $\Gamma$.
Set
\begin{gather}\label{para-F}
\beta_i^F=\sup_{j\in F_i} \beta_j,\quad q_{ii}^F=-\sum_{k\neq i}q_{ik}^F,\\
q_{ik}^F=\left\{\begin{array}{ll} \sup_{r\in F_i}\sum_{j\in F_k} q_{rj},\ &\text{if $k<i$},\\
  \inf_{r\in F_i}\sum_{j\in F_k}q_{rj},\ &\text{if $k>i$}.
 \end{array}\right.
\end{gather}
Then \[\beta_j\leq \beta_i^F, \ \forall\, j\in F_i,\  \text{and}\ \beta_{i-1}^F<\beta_i^F,\ i=2,\ldots,m.\]
After doing these preparation, we can get the following result.
\begin{thm}\label{t-3}
Let $(X_t,\La_t)$ be a state-independent \rsp in $\S$ with $N=\infty$.
Assume that (H) holds and $(\La_t)$ is recurrent. Let $V\in C^2(\R^d)$ such that (A1) is satisfied and $\bar K=\sup_{i\in \S}\beta_i<\infty$. Define the partition $\Gamma$ and the corresponding vector $(\beta_i^F)$, finite matrix $Q^F$ as above. Suppose that the $m\times m$ matrix $-\big(\diag(\beta_1^F,\ldots,\beta_m^F)+Q^F\big)H_m$ is a nonsingular M-matrix,
where
\begin{equation}\label{h-matrix}
H_m=\begin{pmatrix}
1&1&1&\cdots&1\\
0&1&1&\cdots&1\\
0&0&1&\cdots&1\\
\vdots&\vdots&\vdots&\cdots&\vdots\\
0&0&0&\cdots&1
\end{pmatrix}_{m\times m}.
\end{equation}
Then the process $(X_t,\La_t)$ is recurrent if\, $\lim_{|x|\ra \infty}V(x)=\infty$ and is transient if \,$\lim_{|x|\ra \infty} V(x)=0$.
\end{thm}

\begin{proof}
As $-\big(Q^F+\diag(\beta_1^F,\ldots,\beta_m^F)\big)H_m$ is a nonsingular M-matrix, by Proposition \ref{m-matrix}, there exists a vector $\mathbf{\eta}^F=(\eta_1^F,\ldots,\eta_m^F)^\ast\gg 0$ such that
\[\bla^F=(\lambda_1^F,\ldots,\lambda_m^F)^\ast=-\big(Q^F+\diag(\beta_1^F,\ldots,\beta_m^F)\big)H_m\eta^F\gg 0.\]
Hence, $\bar\lambda:=\max_{1\leq i\leq m}\lambda_i^F>0$. Set $\bxi^F=H_m\eta^F$.
Then
\[\xi_i^F=\eta_m^F+\cdots+\eta_i^F,\quad i=1,\ldots,m,\]
which implies that $\xi_{i+1}^F<\xi_i^F$ for $i=1,\ldots,m-1$, and $\bxi^F\gg 0$.
For each $j\in \S$, we define
$\xi_j=\xi_i^F$ if $j\in F_i$, which is reasonable as $(F_i)$ is a finite partition of $\S$. Via this method, we get a vector $\bxi=(\xi_1,\xi_2,\ldots)^\ast$ from $\bxi^F$.

Let $\mathscr I:\S\ra \{1,2,\ldots,m\}$ be a map defined by $\mathscr I(j)=k$ if $j\in F_k$. Let $Q_xg(i)=\sum_{j\neq i}q_{ij}(x)(g_j-g_i)$ for $g\in \mathscr B(\S)$. Set $f(x,r)=V(x)\xi_r$, $x\in \R^d$, $r\in \S$. By the definition of $(\beta^F_i)$ and $Q^F$, we obtain that for $r\in F_i$
\begin{align*}
Q\bxi(r)&=\sum_{j\neq r} q_{rj}(\xi_j-\xi_i)=\sum_{j\not\in F_i}q_{rj}(\xi_j-\xi_i)\\
&=\sum_{k<i}\big(\sum_{j\in F_k}  q_{rj}\big)(\xi_k^F-\xi_i^F)+\sum_{k>i}\big(\sum_{j\in F_k} q_{rj}\big)(\xi_k^F-\xi_i^F)\\
&\leq \sum_{k<i}q_{ik}^F(\xi_k^F-\xi_i^F)+\sum_{k>i}q_{ik}^F(\xi_k^F-\xi_i^F) =Q^F\bxi^F(\proj(r)).
\end{align*}
Moreover,
\begin{align*}
\mathscr A f(x,r)&=Q\bxi(r)V(x)+\xi_rL^{(r)}V(x)\\
&\leq \big(Q^F\bxi^F(\mathscr I(r))+\beta_{\proj(r)}^F\xi_{\proj(r)}^F\big)V(x)\\
&=-\lambda_{\proj(r)}V(x)\leq 0.
\end{align*}
As $(\La_t)$ is recurrent and (H-iv) holds, recurrence of  $(X_t,\La_t)$  is equivalent to the condition that $\p_{x,i}(\tau_{r_0}<\infty)=1$ for some $r_0>0$ and some $x\in \R^d$ with $|x|>r_0$ and some $i\in \S$.
Here
\begin{equation}\label{tau}
\tau_{r_0}:=\inf\big\{t>0; (X_t,\La_t)\in \{x\in\R^d;\ |x|\leq r_0\}\!\times\! \S\big\}.
\end{equation}
Applying It\^o's formula  to $(X_t,\La_t)$ with $(X_0,\La_0)=(x,l)$ satisfying $|x|>r_0$ (cf. \cite{Sko}), we obtain
\begin{equation}\label{ine-2}
\E f(X_{t\wedge\tau_{r_0}},\La_{t\wedge\tau_{r_0}})\!=\!f(x,l)\!+\!\E \int_0^{t\wedge\tau_{r_0}}\!\!\!\!\mathscr Af(X_s,\La_s)\d s
\!\leq\! f(x,l)\!=\!V(x)\xi_l.
\end{equation}

Firstly, consider the case $\lim_{|x|\ra \infty} V(x)=0$. If $\p(\tau_{r_0}<\infty)=1$, then passing $t$ to $\infty$ in (\ref{ine-2}), we get
\[\inf_{\{y:|y|\leq r_0\}} V(y)\leq \E V(X_{\tau_{r_0}})\leq \max_{i,k}\Big(\frac{\xi_i^F}{\xi_k^F}\Big)V(x),\]
as $|X_{\tau_{r_0}}|= r_0$.  We get $\inf_{\{y:|y|\leq r_0\}}V(y)>0$ by the compactness of set $\{y;|y|\leq r_0\}$ and positiveness of function $V$. So, letting $|x|$ tend to $\infty$ in previous inequality, the right hand goes to $0$, but the left hand is strictly bigger than a positive constant, which is a contradiction. Therefore, $\p(\tau_{r_0}<\infty)>0$, and the process $(X_t,\La_t)$ is transient.

Secondly, consider the case $\lim_{|x|\ra\infty} V(x)=\infty$. Introduce another stopping time
\[\tau_K=\inf\{t>0;\ |X_t|\geq K\}.
\]
As the process $(X_t,\La_t)$ is nonexplosive, $\tau_K$ increases to $\infty$ almost surely as $K\ra \infty$. It\^o's formula also yields that
\[\E[ V(X_{t\wedge\tau_K\wedge\tau_{r_0}})\xi_{\La_{t\wedge\tau_K\wedge\tau_{r_0}}}]\leq V(x)\xi_l.\]
Letting $t\ra \infty$, Fatou's lemma implies that
\[\E\big[V(X_{\tau_K\wedge\tau_{r_0}})]\leq \max_{i,k}\Big(\frac{\xi_i^F}{\xi_k^F}\Big)V(x),\]
and hence,
\[\p(\tau_{r_0}\geq \tau_K)\leq \max_{i,k}\Big(\frac{\xi_i^F}{\xi_k^F}\Big)\frac{V(x)}{\inf_{\{y;|y|=K\}} V(y)}.\]
Since $\lim_{|x|\ra \infty} V(x)=\infty$, letting $K\ra \infty$ in the previous inequality, we obtain that $\p(\tau_{r_0}=\infty)\leq 0$. We have completed the proof.
\end{proof}

As an application of Theorem \ref{t-3}, we construct an example of state-independent \rsp in an infinite state space and study its recurrent property.
\begin{exam}
Let $(\La_t)$ be a birth-death process on $\S=\{1,2,\ldots\}$ with $q_{ii+1}\equiv b>0$, $i\geq 1$, and $q_{ii-1}\equiv a>0$, $i\geq 2$. Assume $a\geq b$, then $(\La_t)$ is recurrent
(see \cite[Table 1.4, p.15]{Chen2}).
Let $X_t$ be a \rsp on $[0,\infty)$ with reflecting boundary at $0$  and satisfies
\[\d X_t=\beta_{\La_t}X_t\d t+\sqrt{2}\d B_t,
\]
where $\beta_i=\kappa-i^{-1}$  for $i\geq 1$.

First, set $V(x)=x$. Let us take the finite partition $F=\{F_1,F_2\}$ to be $F_1=\{1\}$ and $F_2=\{2, 3,\ldots\}$.
It is easy to see that $q_{12}^F=b$ and $q_{21}^F=a$.  Then \[ L^{(i)}V(x)= \beta_i V(x),\quad x>1, \ i\geq 1.\]
So $\beta_1^F=\kappa-1$, $\beta_2^F=\kappa $, and
\[-\big(Q^F+\diag(\beta_1^F,\beta_2^F)\big)H_2=\begin{pmatrix} b-\beta_1^F&-\beta_1^F\\ -a&-\beta^F_2\end{pmatrix}.\]
Applying Proposition \ref{m-matrix}, we get that previous matrix is a nonsingular M-matrix if and only if
\begin{equation}\label{e-e-1}
\kappa <\frac{a+b+1-\sqrt{(a+b+1)^2-4a}}{2}.
\end{equation}
Therefore, according to Theorem \ref{t-3}, if (\ref{e-e-1}) holds, the process $(X_t,\La_t)$ is recurrent.

Second, set $V(x)=x^{-1}$.  We still take $F_1=\{1\}$, and $F_2=\{2,3,\ldots\}$.
Then
\[L^{(i)}V(x)=(-\beta_i+2x^{-2})x^{-1}\leq (-\beta_i+2r_0^{-2})V(x),\quad \text{for}\ x>r_0.\]
Therefore, in this case, $\beta_1^F=-\kappa+1+2r_0^{-2}$ and $\beta_2^F=-\kappa+\frac 12+2r_0^{-2}$.
Set  \begin{equation}\label{e-e-2}\kappa_0=\begin{cases}
  1-b&\text{if}\ 2ab\leq 1-b,\\
  \frac{1-b-a+\sqrt{(a+b-1)^2+4a+2b-2}}{2}&\text{if}\ 2ab>1-b.
\end{cases} \end{equation}
If
$\kappa>\kappa_0$,
then there exist $r_0>0$ such that the
matrix $-\big(Q^F+\diag(\beta_1^F,\beta_2^F))H_2$ is a nonsingular M-matrix.
Consequently, Theorem \ref{t-3} yields that $(X_t,\La_t)$ is transient if $\kappa>\kappa_0$.
More precisely,
if we take $b=1$ and $a=2$, then $(\La_t)$ is exponentially ergodic, but $(X_t,\La_t)$ is transient if $\kappa>\sqrt{3}-1\approx 0.732$ and is recurrent if $\kappa<2-\sqrt{2}\approx0.586$.

Next, we divide the state space into three parts. Precisely, let $F=\{F_1, F_2, F_3\}$ with  $F_1=\{1\}$, $F_2=\{2\}$ and $F_3=\{3,4,\ldots\}$. Corresponding to this partition, we have
\[Q^F=\begin{pmatrix}
  -b&b&0\\ a&-(a+b)&b\\ 0&a&-a
\end{pmatrix}.\] By taking $V(x)=x$ again, we get that $\beta_1^F=\kappa-1$, $\beta_2^F=\kappa-\frac 12$, $\beta_3^F=\kappa$. Consider only the case $b=1$, $a=2$. By Theorem \ref{t-3}, if the matrix
\[-\big(\diag(\beta_1^F,\beta_2^F,\beta_3^F)+Q^F\big)H_3=
\begin{pmatrix}
  2-\kappa&1-\kappa&1-\kappa\\
  -2&\frac 32-\kappa &\frac 12-\kappa\\
  0&-2&-\kappa
\end{pmatrix}\]
is a nonsingular M-matrix, then the process $(X_t,\La_t)$ is recurrent. Applying Proposition \ref{m-matrix}, we obtain that if $\kappa<\frac{1}{4}(11-\sqrt{73})\approx 0.614$, then the previous matrix is a nonsingular M-matrix, and hence $(X_t,\La_t)$ is recurrent. This shows that the upper bound for recurrence of this process can be improved by dividing $\S$ into more pieces.
For the transience, we take $V(x)=x^{-1}$, then
\[L^{(i)}V(x)=(-\beta_i+2x^{-2})x^{-1}\leq (-\beta_i+2r_0^{-2})V(x),\quad \text{for}\ x>r_0.\]
We have $\beta_1^F=-\kappa+1+2r_0^2$, $\beta_2^F=-\kappa+\frac 12+2r_0^{-2}$, and $\beta_3^F=-\kappa+2r_0^{-2}$. Direct calculation yields that if $\kappa>\frac 14(\sqrt{17}-1)\approx 0.7807$,  the matrix $-\big(\diag(\beta_1^F,\beta_2^F,\beta_3^F)+Q^F\big)H_3$ is a nonsingular M-matrix, and hence $(X_t,\La_t)$ is transient due to Theorem \ref{t-3}. Unfortunately, this lower bound is bigger than $\sqrt 3-1$ obtained previously when we just divide $\S$ into two parts.
\end{exam}

Based on Example 2.1, we construct another example of state-dependent \rsp in a finite state space.
\begin{exam}
  Let $X_t$ satisfy the following SDE on $[0,\infty)$ with reflecting boundary at $0$,
  \[\d X_t=\beta_{\La_t}X_t\d t+\sqrt{2}\d B_t,\]
  where $\beta_1=\kappa-1$, $\beta_2=\kappa$, and $(\La_t)$ is a stochastic process on $\S=\{1,2\}$ satisfying
  \[q_{12}(x)=\frac{b(1+2x)}{1+x}, \quad q_{21}(x)=\frac{a(1+2x)}{2(1+x)},\quad x\geq 0,\]
  and  $q_{11}(x)=-q_{12}(x)$, $q_{22}(x)=-q_{21}(x)$. By \eqref{tilde-q}, it is easy to see that
  $\tilde q_{12}=b$ and $\tilde q_{21}=a$.   For the recurrence, we take $V(x)=x$, then
  \[L^{(i)}V(x)=\tilde \beta_i V(x),\quad x>1,\  i=1,2,\]
  where $\tilde \beta_1=\kappa-1$ and $\tilde \beta_2=\kappa$. Therefore, if (\ref{e-e-1}) holds, then $-(\tilde Q+\diag(\tilde \beta_1,\tilde \beta_2))H_2$ is a nonsingular M-matrix, and hence the process $(X_t,\La_t)$ is recurrent due to Theorem \ref{t-finite}.  For the transience, we take $V(x)=x^{-1}$, then
  \[ L^{(i)}V(x)=(-\beta_i+2x^{-2})x^{-1}\leq (-\beta_i+2r_0^{-2})V(x),\quad \text{for}\ x>r_0.\]
  Similar to the discussion in Example 2.1, we obtain that
  the process $(X_t,\La_t)$ is recurrent if  $\kappa>\kappa_0$, where $\kappa_0$ is given by \eqref{e-e-2}.
  \end{exam}

Next, we consider the Ornstein-Uhlenbeck type process with regime-switching, that is,  the process $(X_t,\La_t)$ satisfies:
\begin{equation}\label{2.1}
\d X_t=b_{\La_t}X_t\d t+\sigma_{\La_t}\d B_t,\quad X_0=x\in \R^d,
\end{equation}
where $\sigma_i$ is a $d\times d$ matrix, $b_i$ is a constant, $(B_t)$ is a Brownian motion in $\R^d$, and $(\La_t)$ is a continuous Markov chain on the space $\S=\{1,\ldots,N\}$ with $N<\infty$. Assume that $(\La_t)$, $(B_t)$ are mutually independent. The $Q$-matrix $(q_{ij})$ of $(\La_t)$ is independent of $(X_t)$, and is irreducible and conservative. We assume that the matrix $\sigma_i\sigma_i^\ast$ is positive definite for every $i\in \S$. Let $\mu=(\mu_i)$ be the invariant probability measure of $(\La_t)$. In \cite{GIY}, the authors showed that when $\sum_{i\in\S}\mu_ib_i<0$, the process $(X_t,\La_t)$ is ergodic in weak topology, that is, the distribution of $(X_t,\La_t)$ converges weakly to a probability measure $\nu$. In \cite{BGM, SY}, the tail behavior of $\nu$ was studied.
\begin{prop}\label{t-OU}
Let $(X_t,\La_t)$ be defined by (\ref{2.1}) with $\sigma_i\sigma_i^\ast$ being positive definite for each $i\in\S$.
If\, $\dis\sum_{i\in\S}\mu_ib_i<0$, then $(X_t,\La_t)$ is exponentially ergodic.
 If\, $\dis \sum_{i\in \S}\mu_ib_i>0$, then $(X_t,\La_t)$ is transient.
\end{prop}

\begin{proof}
By (\ref{2.1}), the generator $L^{(i)}$ of $(X_t^{(i)})$ is given
by \[L^{(i)}=\frac12 \sum_{k,l=1}^d a_{kl}^{(i)}\frac{\partial^2}{\partial x_k\partial x_l}+\sum_{k=1}^db_ix_k\frac{\partial}{\partial x_k},\]
where $a^{(i)}=\sigma_i\sigma_i^\ast$.
Take $V(x)=|x|$, then for each $i\in \S$,
\[L^{(i)}V(x)= b_i|x|,\ \text{for $|x|>1$}. \]
As $\dis\lim_{|x|\ra \infty}|x|=\infty$, by Theorem \ref{t-2}, we get $(X_t,\La_t)$ is exponentially ergodic if $\dis\sum_{i\in\S}\mu_ib_i<0$.

Now we take $V(x)=|x|^{-\gamma}$ with $\gamma>0$.
We have
\begin{align*}
L^{(i)}V(x)&=\frac{\gamma(\gamma+2)}{2}\sum_{k,l}a_{kl}^{(i)}|x|^{-\gamma-4}x_kx_l-\frac{\gamma}{2}\big(\sum_{k}a_{kk}^{(i)}\big)|x|^{-\gamma-2}-\gamma b_i|x|^{-\gamma}\\
&=|x|^{-\gamma}\Big(-\gamma b_i+\frac{\gamma(\gamma+2)}{2}|x|^{-4}\sum_{k,l}a_{kl}^{(i)}x_kx_l-\frac{\gamma}{2}|x|^{-2}\sum_{k}a_{kk}^{(i)}\Big),
\end{align*}
for $|x|>r_0>0$. When $r_0$ is sufficiently large, it is easy to see that
\[\frac{\gamma(\gamma+2)}{2}|x|^{-4}\sum_{k,l}a_{kl}^{(i)}x_kx_l-\frac{\gamma}{2}|x|^{-2}\sum_{k}a_{kk}^{(i)}\leq \frac{1}{r_0},\quad \forall\,|x|>r_0.\]
Therefore, we get
\begin{equation}\label{ine-3}
L^{(i)}V(x)\leq (-\gamma b_i+\frac{1}{r_0})V(x),\quad |x|>r_0.
\end{equation}
By Theorem \ref{t-2}, as $\dis\lim_{|x|\ra \infty}|x|^{-\gamma}=0$,
if \[ \sum_{i=1}^N\mu_i(-\gamma b_i+\frac{1}{r_0})=-\gamma\sum_{i=1}^N\mu_i b_i+\frac{1}{r_0}\leq 0,\]
then $(X_t,\La_t)$ is transient.
When $\sum_{i=1}^N\mu_ib_i>0$, we can always find a constant $r_0>0$ sufficiently large such that
$-\gamma\sum_{i=1}^N\mu_ib_i+\frac1{r_0}<0$. Hence, when $\sum_{i=1}^N\mu_ib_i>0$, $(X_t,\La_t)$ is transient.
\end{proof}

\section{Criteria for transience and recurrence: II}

According to Foster-Lyapunov drift condition for diffusion processes, if there exists a function $V\in C^2(\R^d)$ satisfying (A1) with $\beta_i\leq 0$ and $\lim_{|x|\ra \infty} V(x)=\infty$, then the diffusion process $(X_t^{(i)})$ is exponentially ergodic. When there is no diffusion process $(X_t^{(i)})$, $i\in\S$, being exponentially ergodic, we can not find suitable function $V\in C^2(\R^d)$ satisfying (A1), so the criteria introduced in Section 2 are useless for this kind of \rsp.  For example, the diffusion process corresponding to $L^{(i)}=\frac 12\Delta+|x|^\delta\hat b(x/|x|,i)\cdot\nabla$ with $\delta\in [0,1)$ is not exponentially ergodic. Therefore, to deal with this kind of processes, we need to extend our criteria established in Section 2. Let $(X_t,\La_t)$ be defined by (\ref{1.1}) and (\ref{1.2}) and $(X_t^{(i)})$ be the corresponding diffusion process in the fixed environment $i\in\S$ with the generator $L^{(i)}$.  Instead of finding one function $V$ satisfying condition (A1), we look for two functions $h$, $g\in C^2(\R^d)$ satisfying the following condition:
\begin{itemize}
\item[(A2)] There exists some constant $r_0>0$ such that for each $i\in \S$,
\begin{gather*}h(x),\ g(x)>0,\quad L^{(i)}h(x)\leq \beta_ig(x),\quad\forall\ |x|>r_0,\\
 \lim_{|x|\ra \infty}\frac{g(x)}{h(x)}= 0,\ \lim_{|x|\ra \infty} \frac{L^{(i)} g(x)}{g(x)}= 0.
\end{gather*}
\end{itemize}

\begin{thm}\label{t-4}
Let $(X_t,\La_t)$ be a state-independent \rsp defined by (\ref{1.1}) and (\ref{1.2}) with $N<\infty$. Assume (H) holds. Let $\mu$ be the invariant probability measure of the process $(\La_t)$. Suppose that there exist two functions $h,\,g\in C^2(\R^d)$ such that  (A2) holds and
\[\sum_{i=1}^N\mu_i\beta_i<0.\]
Then $(X_t,\La_t)$ is recurrent if\, $\dis\lim_{|x|\ra \infty} h(x)=\infty$ and is transient if $\dis\lim_{|x|\ra \infty} h(x)=0$.
\end{thm}

\begin{proof}
As $\sum_{i=1}^N\mu_i\beta_i<0$, by the Fredholm alternative (see \cite[p.434]{PP}), we obtain that there exist a constant $\kappa>0$ and a vector $\bxi$ such that
\begin{equation}\label{3.1.1}Q\bxi(i)=-\kappa-\beta_i,\quad i\in\S.
\end{equation}
Set $f(x,i)=h(x)+\xi_i g(x)$.
We obtain
\begin{equation}\label{ine-4}
\begin{split}
\mathscr A f(x,i)&=L^{(i)} h(x)+\xi_i L^{(i)}g(x)+Q\bxi(i) g(x)\\
&\leq \Big(\beta_i +Q\bxi(i)+\xi_i\frac{L^{(i)}g(x)}{g(x)}\Big)g(x)
\end{split}
\end{equation} for large $|x|$.
By \eqref{3.1.1}, (\ref{ine-4}) and condition (A2), we get
\begin{equation}\label{ine-5}
\mathscr A f(x,i)\leq \Big(-\kappa+\xi_i\frac{L^{(i)}g(x)}{g(x)}\Big)g(x)\leq 0 \quad \text{for large $|x|$}.
\end{equation}
As $N<\infty$,  $\bxi$ is bounded. Since $\lim_{|x|\ra\infty} \frac{g(x)}{h(x)}=0$ and $f(x,i)=(1+\xi_i\frac{g(x)}{h(x)})h(x)$ for $|x|>r_0$, it is easy to see that there exists  $r_1>0$ such that $f(x,i)>0$ for $|x|>r_1$.
In addition, if $\lim_{|x|\ra \infty} h(x)=\infty$, then $\lim_{|x|\ra \infty} f(x,i)=\infty$; if $\lim_{|x|\ra\infty} h(x)=0$, then $\lim_{|x|\ra \infty} f(x,i)=0$. By the method of Lyapunov function, inequality (\ref{ine-5}) yields that $(X_t,\La_t)$ is recurrent if $\lim_{|x|\ra \infty} h(x)=\infty$ and is transient if $\lim_{|x|\ra \infty} h(x)=0$.
\end{proof}

We proceed to extend the previous criterion to deal with state-dependent \rsp in a finite state-space.

\begin{thm}\label{t-6}
Let $(X_t,\La_t)$ be a state-dependent \rsp in $\S$ with $N<\infty$. Assume (H) holds and there exist functions $h$, $g\in C^2(\R^d)$ such that (A2) holds. Let $\tilde Q$ be defined by (\ref{tilde-q}). Suppose there exists a positive nonincreasing function $\eta$ on $\S$ such that
$\beta_i+\tilde Q\eta(i)<0$ for every $i\in\S$. Then
$(X_t,\La_t)$ is recurrent if \,$\lim_{|x|\ra \infty}h(x)=\infty$ and is transient if \,$\lim_{|x|\ra \infty} h(x)=0$.
\end{thm}

\begin{proof}
  By the nonincreasing property of $\eta$ and the definition of $\tilde Q$, it is easy to see that
  \[Q_x\eta(i)\leq \tilde Q\eta(i),\quad x\in \R^d,\ i\in \S.\]
  Set $f(x,i)=h(x)+\eta_i g(x)$, $x\in \R^d$, $i\in \S$. We get
  \[\mathscr Af(x,i)\leq \big(\beta_i+\tilde Q\eta(i)+\eta_i\frac{L^{(i)} g(x)}{g(x)}\big)g(x).\]
  Then the desired result follows from the same deduction as in the proof of Theorem \ref{t-4}.
\end{proof}

Now we apply our second type criterion on recurrence to investigate the recurrent property of the process studied by \cite{PP}.  Let
\begin{equation}\label{4.1}
\d X_t=|X_t|^\delta\hat b(X_t/|X_t|,\La_t)\d t+\sigma(X_t,\La_t)\d B_t, \ \ X_0=x\in \R^d,\ d\geq 1,
\end{equation}
where $\delta\in [-1,1)$, $\hat b(\cdot,\cdot):S^{d-1}\times \S\ra \R^d$, $\sigma(\cdot,\cdot):\R^d\times \S\ra \R^{d\times d}$, and $(B_t)$ is a $d$-dimensional Brownian motion. Let $(\La_t)$ be a continuous time Markov chain on $\S$ with irreducible conservative $Q$-matrix $(q_{ij})$, which is independent of $(B_t)$. Let $\mu$ be the invariant probability measure of $(\La_t)$. Set $a^{(i)}(x)=\sigma(x,i)\sigma(x,i)^\ast$, which is assumed to be uniformly positive definite.
Suppose condition (H) is satisfied. In \cite{PP}, the authors considered the recurrent property of $(X_t,\La_t)$ under the condition
\[\sum_{i\in \S}\mu_i\hat b(\phi,i)=0,\quad \forall\,\phi\in S^{d-1}.
\]
In this section, we shall study the case $\sum_{i\in \S}\mu_i\hat b(\phi,i)\neq 0$.
\begin{thm}\label{t-5}
Assume that $\|a^{(i)}(\cdot)\|$ is bounded on $\R^d$ for every $i\in \S$. Let
\begin{equation}\label{beta-1}
\beta_i=\left\{\begin{array}{ll} \limsup\limits_{|x|\ra \infty} \sum\limits_{k=1}^d\hat b_k(\frac{x}{|x|},i)\frac{x_k}{|x|}, &\text{if $\delta\in (-1,1)$},\\
                       \limsup\limits_{|x|\ra \infty}\Big( \frac 12\sum_{k=1}^d a_{kk}^{(i)}(x)-\frac{\sum\limits_{k,l=1}^d\!a_{kl}^{(i)}x_kx_l}{2|x|^{2}}+\frac{\sum_{k=1}^d\hat b_k(\frac{x}{|x|},i) x_k}{|x| }\Big),& \text{if $\delta=-1$},
                       \end{array}\right.
\end{equation}
and
\begin{equation}\label{beta-2}
\tilde{\beta}_i=\left\{\begin{array}{ll} \liminf\limits_{|x|\ra \infty} \sum\limits_{k=1}^d\hat b_k(\frac{x}{|x|},i)\frac{x_k}{|x|}, &\text{if $\delta\in (-1,1)$},\\
      \liminf\limits_{|x|\ra \infty}\Big( \frac 12\sum_{k=1}^d a_{kk}^{(i)}(x)-\frac{\sum\limits_{k,l=1}^d\!a_{kl}^{(i)}x_kx_l}{2|x|^{2}}+\frac{\sum_{k=1}^d\hat b_k(\frac{x}{|x|},i) x_k}{|x| }\Big),& \text{if $\delta=-1$}.
                       \end{array}\right.
\end{equation}
If \,$\sum_{i\in\S}\mu_i\beta_i<0$, then $(X_t,\La_t)$ defined by (\ref{4.1}) is recurrent. If\, $\sum_{i\in \S}\mu_i\tilde \beta_i>0$, then $(X_t,\La_t)$ is transient.
\end{thm}

\begin{proof}
For the recurrence, set $h(x)=|x|^\gamma$, $\gamma>0$, and $g(x)=|x|^{\gamma+\delta-1}$.
Then it holds that
\[\lim_{|x|\ra \infty} \frac{g(x)}{h(x)}=0,\quad \lim_{|x|\ra\infty}\frac{L^{(i)} g(x)}{g(x)}=0.\]
By direct calculation we get
\begin{equation}\label{ine-6}
L^{(i)} h(x)=\Big[(\gamma-1)\frac{\sum_{k,l=1}^d a_{kl}^{(i)}(x)x_kx_l}{2|x|^{\delta+3}}+\frac{\sum_{k=1}^da_{kk}^{(i)}(x)}{2|x|^{\delta+1}}+\frac{\sum_{k=1}^d \hat b_k(\frac{x}{|x|},i)x_k}{|x|}\Big]\gamma g(x).
\end{equation}

When $\delta\in (-1,1)$,
\[L^{(i)}h(x)=\Big[O(|x|^{-\delta-1})+\frac{\sum_{k=1}^d \hat b_k(\frac{x}{|x|},i)x_k}{|x|}\Big]\gamma g(x),\]
which implies that if $\sum_{i\in\S}\mu_i\beta_i<0$, then there exists $r_0>0$ such that
\[\sum_{i\in \S}\mu_i\Big(O(|x|^{-\delta-1})+\frac{\sum_{k=1}^d \hat b_k(\frac{x}{|x|},i)x_k}{|x|}\Big)<0,\quad \text{for $|x|>r_0$.}\]
Applying Theorem \ref{t-4}, as $\gamma>0$, we obtain that $(X_t,\La_t)$ is recurrent if $\sum_{i\in\S}\mu_i\beta_i<0$.

When $\delta=-1$,
it holds
\[\limsup_{|x|\ra \infty}\lim_{\delta\downarrow 0}(\gamma-1)\frac{\sum_{k,l=1}^d a_{kl}^{(i)}(x)x_kx_l}{2|x|^{\delta+3}}+\frac{\sum_{k=1}^da_{kk}^{(i)}(x)}{2|x|^{\delta+1}}+\frac{\sum_{k=1}^d \hat b_k(\frac{x}{|x|},i)x_k}{|x|}=\beta_i.\]
Therefore, if $\sum_{i\in\S}\mu_i\beta_i<0$, by choosing $\gamma>0$ sufficiently small and $r_0>0$ sufficiently large, we can use Theorem \ref{t-4} to show that $(X_t,\La_t)$ is recurrent.

For the transience,  set $h(x)=|x|^{-\gamma}$ and $g(x)=|x|^{-\gamma+\delta-1}$ for $\gamma>0$.
Then it still holds
\[\lim_{|x|\ra \infty} \frac{g(x)}{h(x)}=0,\quad \lim_{|x|\ra\infty}\frac{L^{(i)} g(x)}{g(x)}=0,\]
and
\begin{equation}\label{ine-7}
\begin{split}
&L^{(i)} h(x)=\Big[-(\gamma+1)\frac{\sum_{k,l=1}^d a_{kl}^{(i)}(x)x_kx_l}{2|x|^{\delta+3}}\\ &\hspace{2cm} +\frac{\sum_{k=1}^da_{kk}^{(i)}(x)}{2|x|^{\delta+1}}+\frac{\sum_{k=1}^d \hat b_k(\frac{x}{|x|},i)x_k}{|x|}\Big](-\gamma) g(x).
\end{split}
\end{equation}
Note that it is $-\gamma<0$ before $g(x)$ in above equality. Similar to the argument in step (1), we can conclude the proof.
\end{proof}

When the dimension $d$ is equal to 1, we can obtain a complete criterion presented as follows.
\begin{cor}\label{t-delta}
Let $(X_t,\La_t)$ be a regime-switching diffusion on $[0,\infty)$ with reflecting boundary at $0$, where $(X_t)$ satisfies
\[\d X_t=b_{\La_t}X_t^\delta\d t+\sigma_{\La_t}\d B_t, \quad \delta\in [-1,1),\]
where $b_i,\,\sigma_i$ are constants for $i$ in a finite set $\S$. $(\La_t)$ is a continuous time Markov chain on $\S$ independent of $(B_t)$. Then $(X_t,\La_t)$ is recurrent if and only if \,$\sum_{i\in\S} \mu_ib_i\leq 0$.
\end{cor}
\begin{proof}
By taking $h(x)$ and $g(x)$ as in the Theorem \ref{t-5}, it is easy to check that $\beta_i=\tilde \beta_i=b_i$.  So according to Theorem \ref{t-5}, $(X_t,\La_t)$ is recurrent if $\sum_{i\in \S}\mu_ib_i<0$ and is transient if $\sum_{i\in\S}\mu_ib_i>0$. Therefore, we only need to consider the case $\sum_{i\in\S}\mu_ib_i=0$. To deal with this situation, we have to consider it separately according to the range of $\delta$.
Note that it holds $\sum_{i\in\S}\mu_i b_i(Q^{-1}b)(i)<0$ as $\sum_{i\in\S}\mu_ib_i=0$ (cf. \cite{PP}).

Case 1:\ $\delta\in (0,1)$.
For $p>0$, set \[f(x,i)=x^p-p(Q^{-1}b)(i)x^{p-1+\delta}+c_ix^{p-2+2\delta},\]
where the vector $(c_i)$ would be determined later.
By noting that $\delta\in (0,1)$, we obtain
\begin{align*}
\mathscr Af(x,i)&=\big[-p(p-1+\delta)b_i(Qb)(i)+Qc(i)\big] x^{p-2+2\delta}+o(x^{p-2+2\delta}).
\end{align*}
Take $p\in (0,1-\delta)$, then
$\sum_{i\in \S}p(p-1+\delta)\mu_i b_i(Q^{-1}b)(i)>0$. By the Fredholm alternative, there exist a constant $\beta>0$ and a vector $(c_i)$ such that
$Qc(i)=p(p-1+\delta)b_i(Q^{-1}b)(i)-\beta$. Choosing these $p$ and $(c_i)$, we have
$\mathscr Af(x,i)=-\beta x^{p-2+2\delta}+o(x^{p-2+2\delta})$. As $\lim_{|x|\ra \infty} f(x,i)=\infty$ for each $i\in \S$, we obtain that $(X_t,\La_t)$ is recurrent when $\sum_{i\in\S}\mu_ib_i=0$ and $\delta\in (0,1)$.

Case 2:\ $\delta\in [-1,0)$.  In this situation, we take $f(x,i)=x^p-p(Q^{-1}b)(i) x^{p-1+\delta}$. Then
\begin{align*}\mathscr A f(x,i)&=\frac 12 \sigma_i^2p(p\!-\!1)x^{p-2}\!-\!p(Q^{-1}b)(i)\big[\frac 12 (p\!-\!1\!+\!\delta)(p\!-\!2\!+\!\delta)\sigma_i^2x^{p-3+\delta}\\ &\quad +\!b_i(p\!-\!1\!+\!\delta)x^{p-2+2\delta}\big]\\
&=\frac 12\sigma_i^2p(p-1)x^{p-2}+o(x^{p-2}).
\end{align*}
By setting $p\in (0,1)$, we have $\lim_{x\ra \infty}f(x,i)=\infty$ and $\mathscr A f(x,i)\leq 0$. Hence, $(X_t,\La_t)$ is recurrent.

Case 3:\ $\delta=0$. We take $f(x,i)=x^p-p(Q^{-1}b)(i) x^{p-1}+c_i x^{p-2}$. Then
\begin{align*}
\mathscr Af(x,i)&=\big[\frac 12\sigma_i^2 p(p-1)-p(p-1)b_i(Q^{-1}b)(i)+Qc(i)\big]x^{p-2}+o(x^{p-2}).
\end{align*}
Putting $p\in (0,1)$, as
$p(p-1)\sum_{i\in \S}\mu_i\big(\sigma_i^2-b_i(Q^{-1}b)(i)\big)<0$, there exist a vector $(c_i)$ and a positive constant $\beta$ such that
\[Qc(i)+\frac 12\sigma_i^2 p(p-1)-p(p-1)b_i(Q^{-1}b)(i)=-\beta<0.\]
Therefore, we get $\mathscr Af(x,i)\leq 0$ and $\lim_{x\ra\infty}f(x,i)=\infty$, which implies that $(X_t,\La_t)$ is recurrent. We complete the proof.
\end{proof}

\end{document}